\documentclass[a4paper,11pt]{scrartcl}
\usepackage{amsmath,amssymb}
\usepackage{amsthm}
\usepackage{graphicx}
\usepackage[square,sort&compress,numbers]{natbib}
\usepackage{type1cm}
\usepackage{mathtools}\mathtoolsset{showonlyrefs=false}
\usepackage{algpseudocode,algorithm}
\usepackage{subfig}

\newcommand{\Authornote}{\renewcommand{\thefootnote}{\fnsymbol{footnote}}}
\newcommand{\authornote}{\Authornote\footnote}

\theoremstyle{plain}

\newtheorem{lemma}{Lemma}[section]
\theoremstyle{definition}

\newtheorem{example}{Example}[section]
\theoremstyle{remark}
\newtheorem{remark}{Remark}[section]

\newcommand{\reflmm}[1]{Lemma~\ref{#1}}
\newcommand{\refex}[1]{Example~\ref{#1}}
\newcommand{\refalg}[1]{Algorithm~\ref{#1}}

\newcommand{\reffig}[1]{Figure~\ref{#1}}

\newcommand{\finbox}{\nolinebreak\hfill{\small $\blacksquare$}}

\newcounter{alnum}

\newenvironment{Problem}{\begin{array}.{*{20}{l}}\}}{\end{array}}

\newcommand{\prox}{\mathop{\boldsymbol{\mathsf{prox}}}\nolimits}
\newcommand{\argmin}{\operatornamewithlimits{\mathrm{arg\,min}}}
\newcommand{\argmax}{\operatornamewithlimits{\mathrm{arg\,max}}}

\renewcommand{\Re}{\ensuremath{\mathbb{R}}}

\newcommand{\overRe}{\ensuremath{\mathbb{R}\cup\{+\infty\}}}

\newcommand{\bi}[1]{\ensuremath{\boldsymbol{#1}}}

\newcommand{\rr}[1]{\ensuremath{\mathrm{#1}}}
\newcommand{\bs}[1]{\ensuremath{\boldsymbol{\mathsf{#1}}}}

\newcommand{\pdif}[2]{\frac{\partial #1}{\partial #2}}

\newcommand{\SC}{\ensuremath{\mathcal{S}}}

\begin{document}

\begin{center}
  {\Large\bfseries\sffamily%
  A note on a family of proximal gradient methods  }\\
  \medskip
  {\Large\bfseries\sffamily%
  for quasi-static incremental problems in elastoplastic analysis 
  }%
  \par%
  \bigskip%
  {
  Yoshihiro Kanno~\authornote[2]{%
    Mathematics and Informatics Center, 
    The University of Tokyo, 
    Hongo 7-3-1, Tokyo 113-8656, Japan.
    E-mail: \texttt{kanno@mist.i.u-tokyo.ac.jp}. 
    }
  }
\end{center}

\begin{abstract}
  Accelerated proximal gradient methods have recently been developed 
  for solving quasi-static incremental problems of elastoplastic analysis 
  with some different yield criteria. 
  It has been demonstrated through numerical experiments that these 
  methods can outperform conventional optimization-based approaches in 
  computational plasticity. 
  However, in literature these algorithms are described individually for 
  specific yield criteria, and hence there exists no guide for 
  application of the algorithms to other yield criteria. 
  This short paper presents a general form of algorithm design, 
  independent of specific forms of yield criteria, that unifies the 
  existing proximal gradient methods. 
  Clear interpretation is also given to each step of the presented 
  general algorithm so that each update rule is linked to the underlying 
  physical laws in terms of mechanical quantities. 
\end{abstract}

\begin{quote}
  \textbf{Keywords}
  \par
  Elastoplastic analysis; 
  incremental problem; 
  nonsmooth convex optimization; 
  first-order optimization method; 
  proximal gradient method. 
\end{quote}

\section{Introduction}

In these 15 years, it has become a trend to apply constrained convex 
optimization approaches to diverse problems in computational plasticity. 
Particularly, {\em second-order cone programming\/} (SOCP) and 
{\em semidefinite programming\/} (SDP) have drawn considerable attention; 
see \cite{BMP05,YK12,Mak06,MM06,MM07,KLS07,KL12} for SOCP approaches, 
\cite{Bis07,BP07,GHdB14,YK16,KLS08,MM08,Mak10} for SDP approaches, and 
\cite{SA18} for recent survey on numerical methods in plasticity. 
Most of these approaches make use of interior-point methods; 
particularly, it is known that primal-dual interior-point methods can 
solve SOCP and SDP problems in polynomial time, and in practice 
require a reasonably small number of iterations \cite{AL12}. 

Recently, {\em accelerated proximal gradient methods\/} 
have been developed for 
solving quasi-static incremental problems in elastoplastic analysis of 
trusses \cite{Kan16} and continua with the von Mises \cite{SK18} and 
Tresca \cite{SK20} yield criteria. 
In contrast to SOCP and SDP approaches, these methods solve 
unconstrained nonsmooth convex optimization formulations of incremental 
problems. 
Through numerical experiments \cite{SK18,SK20}, it has been 
demonstrated that the accelerated proximal gradient methods outperform 
SOCP and SDP approaches using a standard implementation of a primal-dual 
interior-point method. 

The proximal gradient method is a {\em first-order optimization method\/}, 
which uses function values and gradients to update an incumbent solution 
at each iteration; in other words, first-order optimization methods do not 
use the Hessian information of functions. 
In contrast, the interior-point method is a {\em second-order optimization 
method\/}, which makes use of the Hessian information. 
It is usual that first-order methods need only very small computational 
cost per iteration but show slow convergence, compared with second-order 
methods.  
Recently, {\em accelerated\/} versions of first-order methods have been 
extensively studied especially for solving large-scale convex 
optimization problems \cite{GOSB14,CP11,BT09,LS13,OdC15}. 
Such accelerated first-order methods originate \citet{Nes83,Nes04}. 
These methods show locally fast convergence, while computational cost 
per iteration is still very small. 
Also, most of them are easy to implement. 
Particularly, the accelerated proximal gradient method \cite{PB14,BT09} 
has many applications in data science, including 
regularized least-squares problems \cite{DDDm04,TY10}, 
signal and image processing \cite{BT09,CW05}, 
and binary classification \cite{ITT17}. 
This success of the accelerated proximal gradient method in data science 
supports that it can also be efficient especially for large-scale 
problems in computational plasticity, compared with conventional 
second-order optimization methods. 

The idea behind our use of accelerated first-order optimization methods 
for equilibrium analysis of structures can be understood as follows. 
For simplicity, consider static equilibrium analysis of an elastic 
structure. 
Let $\bi{u} \in \Re^{d}$ denote the nodal displacement vector, where $d$ 
is the number of degrees of freedom. 
We use $\pi(\bi{u})$ to denote the elastic energy stored in the 
structure. 
For a specified static external load vector $\bi{q} \in \Re^{d}$, 
the total potential energy is given as 
$\pi(\bi{u}) - \bi{q} \cdot \bi{u}$. 
The equilibrium state is characterized as a stationary point of this 
total potential energy function. 
Application of the steepest descent method (which is a typical 
first-order method) to the minimization problem of 
$\pi(\bi{u}) - \bi{q} \cdot \bi{u}$ results in the iteration 
\begin{align}
  \bi{u}^{(k+1)} 
  := \bi{u}^{(k)} 
  - \alpha (\nabla \pi(\bi{u}^{(k)}) - \bi{q}) , 
  \label{eq:elastic.steepest}
\end{align}
where $\alpha > 0$ is the step length. 
Here, $\nabla \pi(\bi{u}^{(k)}) - \bi{q}$ on the right side is the 
unbalanced nodal force vector (i.e., the residual of the force-balance 
equation) at $\bi{u}^{(k)}$. 
Thus, the computation required for each iteration in 
\eqref{eq:elastic.steepest} is very cheap, compared with an iteration of 
a second-order method, e.g., the Newton--Raphson method 
(which needs to solve a system of linear equations). 
However, it is well known that the steepest descent method spends a 
large number of iterations before convergence: The sequence of objective 
values generated by the steepest descent method converges to the optimal 
value at a linear rate \cite[Theorem~3.4]{NW06}. 
To improve this slow convergence, we may apply Nesterov's acceleration 
to \eqref{eq:elastic.steepest}, which yields the iteration 
\begin{align*}
  \bi{u}^{(k+1)} 
  &:= \bi{v}^{(k)} 
  - \alpha (\nabla \pi(\bi{v}^{(k)}) - \bi{q})  , \\
  \bi{v}^{(k+1)} 
  &:= \bi{u}^{(k+1)}  + \omega_{k} (\bi{u}^{(k+1)} - \bi{u}^{(k)}) , 
\end{align*}
where $\omega_{k} > 0$ is an appropriately chosen 
parameter \cite{Nes83,Nes04}. 
We can see that the computation required for each iteration is still 
very cheap. 
In contrast, the convergence rate is drastically improved: Under several 
assumptions such as strong convexity of $\pi$, local quadratic 
convergence is guaranteed \cite{Nes83,Nes04}. 
In practice, we also incorporate the adaptive restart of 
acceleration \cite{OdC15}, to achieve monotone decrease of 
the objective value. 
Indeed, for elastic problems with material nonlinearity, the numerical 
experiments in \cite{FK19} demonstrate that the accelerated steepest 
descent method outperforms conventional second-order optimization 
methods, especially when the size of a problem instance is large. 

The total potential energy for an elastoplastic incremental problem is 
nonsmooth in general, due to nonsmoothness of the plastic dissipation 
function. 
Therefore, unlike an elastic problem considered above, application of 
the (accelerated) steepest descent method is inadequate. 
Instead, as shown in \cite{Kan16,SK18,SK20}, 
the (accelerated) proximal gradient method is well suited for 
elastoplastic incremental problems. 
In \cite{Kan16,SK18,SK20}, although the concrete steps of the algorithms 
are presented, it is not explained how these steps correspond to the 
underlying physical laws in terms of quantities in mechanics. 
This short paper presents a clearer understanding of 
this correspondence relation. 
Moreover, in \cite{Kan16,SK18,SK20}, for each specific yield criterion 
an algorithm is described individually, and hence comprehensive vision 
of algorithm design is not presented. 
This paper provides a general scheme for algorithm design, independent 
of specific forms of yield criteria. 
With these two contributions, this paper attempts to provide 
a deeper understanding of (accelerated) proximal gradient methods for 
computational plasticity.


The paper is organized as follows. 
Section~\ref{sec:problem} states the elastoplastic incremental problem 
that we consider in this paper. 
Section~\ref{sec:algorithm} presents a general form of the algorithm 
that unifies existing proximal gradient methods for some specific 
yield criteria. 
Section~\ref{sec:understanding} presents an interpretation of this 
general form to provide a clear insight. 
Finally, some conclusions are drawn in section~\ref{sec:conclusion}.

\section{Elastoplastic incremental problem}
\label{sec:problem}

In this section, we formally state the problem considered in this paper. 
Namely, we consider a quasi-static incremental problem of an 
elastoplastic body, where small deformation is assumed. 
Although the proximal gradient methods in \cite{Kan16,SK18,SK20} 
deal with the strain hardening, in this paper we restrict ourselves to 
perfect plasticity (i.e., a case without the strain hardening) 
for the sake of simplicity of presentation. 

Consider an elastoplastic body discretized according to the conventional 
finite element procedure. 
Suppose that we are interested in quasi-static behavior of the 
body in the time interval $[0,T]$. 
This time interval is subdivided into some intervals. 
For a specific subinterval, denoted by $[t, t+\Delta t]$, we apply the 
backward Euler scheme and attempt to find the equilibrium state at 
$t+\Delta t$. 

Let $d$ and $m$ denote the number of degrees of freedom of the nodal 
displacements and the number of the evaluation points of the Gauss 
quadrature, respectively. 
We use $\Delta\bi{u} \in \Re^{d}$ to denote the incremental nodal 
displacement vector. 
At numerical integration point $l$ $(l=1,\dots,m)$, 
let $\Delta\bi{\varepsilon}_{\rr{e}l} \in \SC^{3}$ and 
$\Delta\bi{\varepsilon}_{\rr{p}l} \in \SC^{3}$ denote the incremental 
elastic and plastic strain tensors, respectively, where 
$\SC^{3}$ denotes the set of second-order symmetric tensors with 
dimension three. 
The compatibility relation between the incremental displacement and the 
incremental strain is given as 
\begin{align}
  & \Delta\bi{\varepsilon}_{\rr{e}l} + \Delta\bi{\varepsilon}_{\rr{p}l} 
  = B_{l} \Delta\bi{u} , 
  \quad l=1,\dots,m, 
  \label{eq:equilibrium.1}
\end{align}
where $B_{l}$ is a linear operator. 

Let $\bi{\sigma}_{l} \in \SC^{3}$ $(l=1,\dots,m)$ and 
$\bi{q} \in \Re^{d}$ denote the stress tensor and 
the external nodal load at time $t+\Delta t$. 
The force-balance equation is written as 
\begin{align}
  \sum_{l=1}^{m} \rho_{l} B_{l}^{*} \bi{\sigma}_{l} = \bi{q} , 
  \label{eq:equilibrium.2}
\end{align}
where $B_{l}^{*}$ is the adjoint operator of $B_{l}$, and 
$\rho_{l}$ $(>0)$ is a constant determined from the weight for the 
numerical integration and the volume of the corresponding finite element. 

Let $\bi{\sigma}_{0l} \in \SC^{3}$ denote the (known) stress at time $t$. 
The constitutive equation is written as 
\begin{align}
  \bi{\sigma}_{l} 
  = \bi{\sigma}_{0l} 
  + \bs{C}_{l} \Delta\bi{\varepsilon}_{\rr{e}l} , 
  \quad l=1,\dots,m, 
  \label{eq:equilibrium.3}
\end{align}
where $\bs{C}_{l}$ is the elasticity tensor. 
Let $Y_{l} \subset \SC^{3}$ denote the admissible set of stress, where 
the boundary of $Y_{l}$ corresponds to the yield surface. 
As usual in plasticity, we assume that $Y_{l}$ is a nonempty convex set. 
Let $\delta_{Y_{l}} : \SC^{3} \to \overRe$ denote the 
{\em indicator function\/} of $Y_{l}$, i.e., 
\begin{align*}
  \delta_{Y_{l}}(\bi{\sigma}_{l}) = 
  \begin{dcases*}
    0 
    & if $\bi{\sigma}_{l} \in Y_{l}$, \\
    +\infty 
    & if $\bi{\sigma}_{l} \not\in Y_{l}$. 
  \end{dcases*}
\end{align*}
The postulate of the maximum plastic work is written as \cite{HR13,Kan11} 
\begin{align}
  \bi{\sigma}_{l} 
  \in \argmax_{\check{\bi{\sigma}}_{l} \in \SC^{3}} \{
  \Delta\bi{\varepsilon}_{\rr{p}l} : \check{\bi{\sigma}}_{l} 
  - \delta_{Y_{l}}(\check{\bi{\sigma}}_{l}) 
  \} , 
  \quad
  l=1,\dots,m, 
  \label{eq:equilibrium.4.1}
\end{align}
where $\Delta\bi{\varepsilon}_{\rr{p}l} : \check{\bi{\sigma}}_{l}$ 
is the scalar product (i.e., the double dot product) 
of $\Delta\bi{\varepsilon}_{\rr{p}l}$ and $\check{\bi{\sigma}}_{l}$. 
Let $\delta_{Y_{l}}^{*} : \SC^{3} \to \overRe$ denote the 
{\em conjugate function\/} of $\delta_{Y_{l}}^{*}$, i.e., 
\begin{align*}
  \delta_{Y_{l}}^{*}(\Delta\bi{\varepsilon}_{\rr{p}l}) 
  &= \sup
  \{ \Delta\bi{\varepsilon}_{\rr{p}l} : \bi{\sigma}_{l} 
  - \delta_{Y_{l}}(\bi{\sigma}_{l}) 
  \mid
  \bi{\sigma}_{l} \in \SC^{3} \} \\
  &= \sup
  \{ \Delta\bi{\varepsilon}_{\rr{p}l} : \bi{\sigma}_{l} 
  \mid
  \bi{\sigma}_{l} \in Y_{l} \} , 
\end{align*}
which is called the {\em dissipation function\/} \cite{HR13}.\footnote{%
In convex analysis, $\delta_{Y_{l}}^{*}$ is known as a {\em support function\/} 
of $Y_{l}$ \cite{Roc70}. } 
It is worth noting that $\delta_{Y_{l}}^{*}$ is a closed proper convex 
function. 
We use 
$\partial\delta_{Y_{l}}^{*}(\Delta\bi{\varepsilon}_{\rr{p}l}) \subseteq \SC^{3}$ 
to denote the {\em subdifferential\/} of $\delta_{Y_{l}}^{*}$ at 
$\Delta\bi{\varepsilon}_{\rr{p}l}$, i.e., 
\begin{align*}
  \partial\delta_{Y_{l}}^{*}(\Delta\bi{\varepsilon}_{\rr{p}l}) 
  = \{ \bi{\sigma}_{l} \in \SC^{3} 
  \mid
  & \delta_{Y_{l}}^{*}(\bi{e}) 
  \ge \delta_{Y_{l}}^{*}(\Delta\bi{\varepsilon}_{\rr{p}l}) \\
  &\quad 
  {}+ \bi{\sigma}_{l} : (\bi{e} - \Delta\bi{\varepsilon}_{\rr{p}l})
  \ (\forall \bi{e} \in \SC^{3})
  \} . 
\end{align*}
As a fundamental result of convex analysis, 
\eqref{eq:equilibrium.4.1} is equivalent to\footnote{%
For a closed proper convex function $f : \Re^{n} \to \overRe$, it is 
known that \cite[Proposition~2.1.12]{Kan11} 
\begin{align*}
  \bi{x} 
  \in \argmax_{\bi{x} \in \Re^{n}} \{
  \langle \bi{s}, \bi{x} \rangle - f(\bi{x}) 
  \}
\end{align*}
is equivalent to 
\begin{align*}
  \bi{x} \in \partial f^{*}(\bi{s}) . 
\end{align*}
 }
\begin{align}
  \bi{\sigma}_{l} 
  \in \partial\delta_{Y_{l}}^{*}(\Delta\bi{\varepsilon}_{\rr{p}l}) , 
  \quad l=1,\dots,m. 
  \label{eq:equilibrium.4}
\end{align}

Accordingly, the incremental problem to be solved is formulated as 
\eqref{eq:equilibrium.1}, \eqref{eq:equilibrium.2}, 
\eqref{eq:equilibrium.3}, and \eqref{eq:equilibrium.4}, 
where $\Delta\bi{\varepsilon}_{\rr{e}l}$, 
$\Delta\bi{\varepsilon}_{\rr{p}l}$, $\bi{\sigma}_{l}$ $(l=1,\dots,m)$, 
and $\Delta\bi{u}$ are unknown variables.

\section{General form of proximal gradient method for elastoplastic incremental problems}
\label{sec:algorithm}

For a specific yield criterion, a proximal gradient method was 
individually proposed in each of \cite{Kan16,SK18,SK20}. 
In this section, we present a unified perspective of these methods by 
providing a general form of a proximal gradient method solving a general 
elastoplastic incremental problem in 
\eqref{eq:equilibrium.1}, \eqref{eq:equilibrium.2}, 
\eqref{eq:equilibrium.3}, and \eqref{eq:equilibrium.4}. 
Also, from a mechanical point of view, clear interpretation is given to 
each step of the iteration. 

We begin by observing that $\Delta\bi{\varepsilon}_{\rr{e}l}$ can be 
eliminated by using \eqref{eq:equilibrium.1}. 
Then the increment of the stored elastic energy associated with the 
Gauss evaluation point $l$ $(l=1,\dots,m)$ can be written as 
\begin{align*}
  w_{l}(\Delta\bi{\varepsilon}_{\rr{p}l},\Delta\bi{u}) 
  &= \frac{1}{2} \rho_{l} 
  \bs{C}_{l} (B_{l} \Delta\bi{u} - \Delta\bi{\varepsilon}_{\rr{p}l}) 
  : (B_{l} \Delta\bi{u} - \Delta\bi{\varepsilon}_{\rr{p}l})  \\
  & \qquad 
  {}+ \rho_{l} \bi{\sigma}_{0l} : 
  (B_{l} \Delta\bi{u} - \Delta\bi{\varepsilon}_{\rr{p}l}) . 
\end{align*}
This is a convex quadratic function, because $\bs{C}_{l}$ is a constant 
positive definite tensor. 
Accordingly, the minimization problem of the total potential energy is 
formulated as follows: 
\begin{align}
  \text{Min.\ }
  \quad
  & \sum_{l=1}^{m} w_{l}(\Delta\bi{\varepsilon}_{\rr{p}l},\Delta\bi{u}) 
  + \sum_{l=1}^{m} \rho_{l} 
  \delta_{Y_{l}}^{*}(\Delta\bi{\varepsilon}_{\rr{p}l})  
  - \bi{q} \cdot \Delta\bi{u} . 
  \label{P.unconstrained.1}
\end{align}
Here, $\bi{q} \cdot \Delta\bi{u}$ is the scalar product 
of $\bi{q}$ and $\Delta\bi{u}$. 
The optimality condition of problem \eqref{P.unconstrained.1} corresponds 
to \eqref{eq:equilibrium.1}, \eqref{eq:equilibrium.2}, 
\eqref{eq:equilibrium.3}, and \eqref{eq:equilibrium.4}; 
see, for more accounts with specific yield 
criteria, \cite{Kan16,Kan11,SK18,YK12}. 

For a closed convex function $f:\SC^{3} \to \overRe$, let 
$\prox_{f} : \SC^{3} \to \SC^{3}$ denote 
the {\em proximal operator\/} of $f$, which is defined by 
\begin{align*}
  \prox_{f}(\bi{\chi})
  = \argmin_{\bi{\zeta} \in \SC^{3}} \Bigl\{
  f(\bi{\zeta}) + \frac{1}{2} \| \bi{\zeta} - \bi{\chi} \|_{\rr{F}}^{2} 
  \Bigr\}  
\end{align*}
for any $\bi{\chi} \in \SC^{3}$, 
where $\| \bi{\zeta} - \bi{\chi} \|_{\rr{F}}$ is the 
Frobenius norm of $\bi{\zeta} - \bi{\chi}$, i.e., 
$\| \bi{\zeta} - \bi{\chi} \|_{\rr{F}}
  = \sqrt{ (\bi{\zeta} - \bi{\chi}) : (\bi{\zeta} - \bi{\chi}) }$. 
A proximal gradient method applied to problem \eqref{P.unconstrained.1} 
consists of the iteration 
\begin{align*}
  \Delta\bi{u}^{(k+1)} 
  &:= \Delta\bi{u}^{(k)} 
  - \alpha \Bigl(
  \sum_{l=1}^{m} \pdif{}{\Delta\bi{u}} 
  w_{l}(\Delta\bi{\varepsilon}_{\rr{p}l}^{(k)},\Delta\bi{u}^{(k)}) 
  - \bi{q} 
  \Bigr) , \\
  \Delta\bi{\varepsilon}_{\rr{p}l}^{(k+1)} 
  &:= \prox_{\alpha \rho_{l} \delta_{Y_{l}}^{*}}\Bigl(
  \Delta\bi{\varepsilon}_{\rr{p}l}^{(k)} 
  - \alpha \pdif{}{\Delta\bi{\varepsilon}_{\rr{p}l}} 
  w_{l}(\Delta\bi{\varepsilon}_{\rr{p}l}^{(k)},\Delta\bi{u}^{(k)}) 
  \Bigr) , \\
  & \qquad\qquad\qquad\qquad\qquad l=1,\dots,m . 
\end{align*}
Here, $\alpha > 0$ is a step length. 
For the guarantee of convergence we let $\alpha \le 1/L$ \cite{BT09,PB14}, 
where $L$ denotes the maximum eigenvalue of the Hessian matrix of 
$\sum_{l=1}^{m} w_{l}(\Delta\bi{\varepsilon}_{\rr{p}l},\Delta\bi{u})$. 
A short calculation shows that the iteration above can be written in 
an explicit manner as \refalg{alg:proximal.plasticity}. 

\begin{algorithm}
  \caption{Proximal gradient method for solving 
  problem \eqref{P.unconstrained.1}. }
  \label{alg:proximal.plasticity}
  \begin{algorithmic}[1]
    \Require 
    $\Delta\bi{u}^{(0)} \in \Re^{d}$, 
    $\Delta\bi{\varepsilon}_{\rr{p}l}^{(0)} \in \SC^{3}$ $(l=1,\dots,m)$, 
    $\alpha \in (0, 1/L]$. 
    Set $\beta_{l} \gets \rho_{l} \alpha$ $(l=1,\dots,m)$. 
    \For{$k=0,1,2,\dots$}
    \State \label{alg:proximal.plasticity.elastic}
    $\Delta\bi{\varepsilon}_{\rr{e}l}^{(k)}
    \gets B_{l} \Delta\bi{u}^{(k)} 
    - \Delta\bi{\varepsilon}_{\rr{p}l}^{(k)}$ 
    $(l=1,\dots,m)$. 
    \State \label{alg:proximal.plasticity.stress}
    $\bi{\sigma}_{l}^{(k)} 
    \gets \bi{\sigma}_{0l} + \bs{C}_{l} : \Delta\bi{\varepsilon}_{\rr{e}l}^{(k)}$ 
    $(l=1,\dots,m)$. 
    \State \label{alg:proximal.plasticity.residual}
    $\displaystyle
    \bi{r}^{(k)} 
    \gets \sum_{l=1}^{m} \rho_{l} B_{l}^{*} \bi{\sigma}_{l}^{(k)} - \bi{q}$. 
    \State \label{alg:proximal.plasticity.displacement}
    $\Delta\bi{u}^{(k+1)} 
    \gets \Delta\bi{u}^{(k)} - \alpha \bi{r}^{(k)}$. 
    \State \label{alg:proximal.plasticity.prox}
    $\displaystyle
    \Delta\bi{\varepsilon}_{\rr{p}l}^{(k+1)} 
    \gets \prox_{\beta_{l} \delta_{Y_{l}}^{*}}
    (\Delta\bi{\varepsilon}_{\rr{p}l}^{(k)} + \beta_{l} \bi{\sigma}_{l}^{(k)})$ 
    $(l=1,\dots,m)$. 
    \State 
    Terminate if 
    $\|\Delta\bi{u}^{(k+1)} - \Delta\bi{u}^{(k)}\|$ and 
    $\| \Delta\bi{\varepsilon}_{\rr{p}l}^{(k+1)} 
    - \Delta\bi{\varepsilon}_{\rr{p}l}^{(k)}\|$ $(l=1,\dots,m)$ 
    are small. 
    \EndFor
  \end{algorithmic}
\end{algorithm}

\begin{figure*}[tbp]
  \centering
  \subfloat[]{
  \label{fig:dissipation_function}
  \includegraphics[scale=0.90]{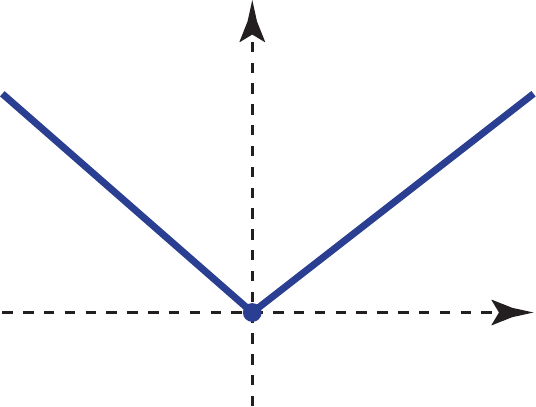}
  \begin{picture}(0,0)
    \put(-150,-50){
    \put(125,64){{\footnotesize $\Delta\varepsilon_{\rr{p}l}$ }}
    \put(63,146){{\footnotesize $z$ }}
    \put(108,136){{\footnotesize $z = \delta_{Y_{l}}^{*}(\Delta\varepsilon_{\rr{p}l})$ }}
    }
  \end{picture}
  }
  \qquad\qquad
  \subfloat[]{
  \label{fig:soft_thresholding}
  \includegraphics[scale=0.90]{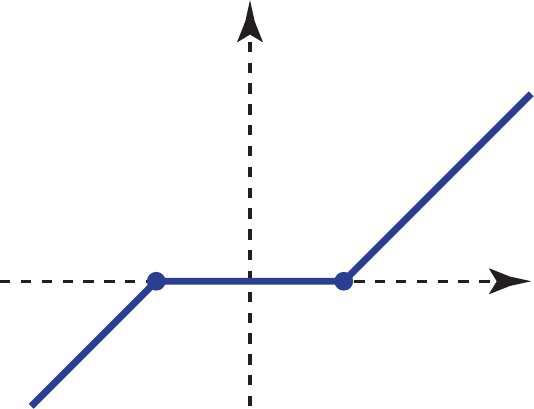}
  \begin{picture}(0,0)
    \put(-150,-50){
    \put(125,72){{\footnotesize $\Delta\varepsilon_{\rr{p}l}$ }}
    \put(63,146){{\footnotesize $z$ }}
    \put(104,138){{\footnotesize $z = \prox_{\beta_{l}\delta_{Y_{l}}^{*}}(\Delta\varepsilon_{\rr{p}l})$ }}
    }
  \end{picture}
  }
  \caption{The dissipation function of \refex{ex:truss} and 
  its proximal operator. }
\end{figure*}

\begin{example}\label{ex:truss}
  As a simple example, consider a truss consisting of $m$ bars. 
  For truss element $l$ $(l=1,\dots,m)$, we use $\sigma_{l}$ and 
  $\Delta\varepsilon_{\rr{p}l}$ to denote the axial stress and the 
  incremental axial plastic strain, respectively. 
  The set of admissible stress is given as 
  \begin{align*}
    Y_{l} = \{
    \sigma_{l} \in \Re
    \mid
    | \sigma_{l} | \le R_{l} 
    \} , 
  \end{align*}
  where $R_{l}$ is the magnitude of the yield stress. 
  Then the dissipation function is 
  \begin{align*}
    \delta_{Y_{l}}^{*}(\Delta\varepsilon_{\rr{p}l}) 
    = R_{l} | \Delta\varepsilon_{\rr{p}l} | , 
  \end{align*}
  which is depicted in \reffig{fig:dissipation_function}. 
  The proximal operator of $\beta_{l}\delta_{Y_{l}}^{*}$, used in 
  line \ref{alg:proximal.plasticity.prox} of 
  \refalg{alg:proximal.plasticity}, is 
  \begin{align*}
    \prox_{\beta_{l}\delta_{Y_{l}}^{*}}(\Delta\varepsilon_{\rr{p}l}) = 
    \begin{dcases*}
      \Delta\varepsilon_{\rr{p}l} +  \beta_{l} R_{l} 
      & if $\Delta\varepsilon_{\rr{p}l} < -  \beta_{l} R_{l}$, \\
      0 
      & if $|\Delta\varepsilon_{\rr{p}l}| \le \beta_{l} R_{l}$, \\
      \Delta\varepsilon_{l} -  \beta_{l} R_{l} 
      & if $\Delta\varepsilon_{\rr{p}l} >  \beta_{l} R_{l}$, \\
    \end{dcases*}
  \end{align*}
  which is depicted in \reffig{fig:soft_thresholding}; 
  see \cite{Kan16} for details. 
  \finbox
\end{example}

Each iteration of \refalg{alg:proximal.plasticity} consists of the 
following procedures. 
In line~\ref{alg:proximal.plasticity.elastic}, according to the 
compatibility relations, we update the incremental elastic strains, 
by using the incumbent incremental displacement and 
the incumbent incremental plastic strains. 
In line~\ref{alg:proximal.plasticity.stress}, we update the stress 
tensors according to the constitutive equations. 
In line~\ref{alg:proximal.plasticity.residual}, we compute the 
unbalanced nodal force vector, $\bi{r}^{(k)}$, according to the 
force-balance equation. 
Line~\ref{alg:proximal.plasticity.displacement} updates the incremental 
displacement vector, by adding $-\alpha \bi{r}^{(k)}$ to the incumbent 
solution, $\Delta\bi{u}^{(k)}$. 
This update rule is analogous to the steepest descent method applied to 
an elastic problem \cite{FK19}; see \eqref{eq:elastic.steepest}. 
Line~\ref{alg:proximal.plasticity.prox} updates the incremental plastic 
strains, where the proximal operator of the dissipation function 
(scaled by $\beta_{l}$) is used. 
This step is further discussed in section~\ref{sec:understanding}. 

The computations of \refalg{alg:proximal.plasticity}, except for the one 
in line~\ref{alg:proximal.plasticity.prox}, consist of additions and 
multiplications, which are computationally very cheap; it is worth 
noting that $B_{l}$ $(l=1,\dots,m)$ is usually sparse. 
For the computation in line~\ref{alg:proximal.plasticity.prox}, 
explicit formulae have been given for the truss \cite{Kan16}, 
von Mises \cite{SK18}, and Tresca \cite{SK20} yield criteria. 
Therefore, for these yield criteria, the computation in 
line~\ref{alg:proximal.plasticity.prox} is also cheap. 
Thus, to extend this algorithm to another yield criterion, it is 
crucial to develop an efficient computational manner for 
line~\ref{alg:proximal.plasticity.prox}. 

It is worth noting that, in practice, we incorporate the acceleration 
scheme \cite{BT09} and its restart scheme \cite{OdC15} into 
\refalg{alg:proximal.plasticity}, to reduce the number of iterations 
required before convergence; see \cite{Kan16,SK18,SK20}. 
Also, every step of \refalg{alg:proximal.plasticity} is highly 
parallelizable.
Namely, the computations in lines \ref{alg:proximal.plasticity.elastic}, 
\ref{alg:proximal.plasticity.stress}, and 
\ref{alg:proximal.plasticity.prox} are carried out independently for 
each numerical integration point, and the vector additions in 
lines \ref{alg:proximal.plasticity.residual} and 
\ref{alg:proximal.plasticity.displacement} can be performed independently 
for each row.

\section{Understanding as fixed-point iteration}
\label{sec:understanding}

This section provides an interpretation of 
\refalg{alg:proximal.plasticity} from a perspective of a fixed-point 
iteration. 
A key is the following property of the proximal operator. 
\begin{lemma}\label{lmm:resolvent.interpretation}
  Let $f : \SC^{3} \to \overRe$ be a closed proper convex function, 
  and $\gamma > 0$. 
  For $\bi{\chi}$, $\bi{\zeta} \in \SC^{3}$, we have 
  $\bi{\zeta} = \prox_{\gamma f}(\bi{\chi})$ if and only if 
  $\bi{\chi} = \bi{\zeta} + \gamma \partial f(\bi{\zeta})$. 
\end{lemma}
\begin{proof}
  See \cite[section~3.2]{PB14}. 
\end{proof}

In accordance with lines \ref{alg:proximal.plasticity.elastic} and 
\ref{alg:proximal.plasticity.stress} of \refalg{alg:proximal.plasticity}, 
we write 
\begin{align*}
  \bi{\sigma}_{l}(\Delta\bi{\varepsilon}_{\rr{p}l},\Delta\bi{u}) 
  = \bi{\sigma}_{0l} 
  + \bs{C}_{l} (B_{l} \Delta\bi{u} - \Delta\bi{\varepsilon}_{\rr{p}l}), 
  \quad 
  l=1,\dots,m 
\end{align*}
for notational simplicity, where 
$\bi{\sigma}_{l}(\Delta\bi{\varepsilon}_{\rr{p}l},\Delta\bi{u})$ is the 
incumbent stress determined from $\Delta\bi{\varepsilon}_{\rr{p}l}$ and 
$\Delta\bi{u}$. 
It follows from \eqref{eq:equilibrium.2} that the unbalanced nodal force 
vector (i.e., the residual of the force-balance equation) can be written as 
\begin{align*}
  & \bi{r}(\Delta\bi{u},\Delta\bi{\varepsilon}_{\rr{p}1},
  \dots,\Delta\bi{\varepsilon}_{\rr{p}m}) 
  = \sum_{l=1}^{m} \rho_{l} B_{l}^{*} 
  \bi{\sigma}_{l}(\Delta\bi{\varepsilon}_{\rr{p}l},\Delta\bi{u}) 
  - \bi{q} . 
\end{align*}
Then the system of \eqref{eq:equilibrium.1}, 
\eqref{eq:equilibrium.2}, \eqref{eq:equilibrium.3}, and 
\eqref{eq:equilibrium.4} is equivalently rewritten as 
\begin{align}
  \bi{r}(\Delta\bi{u},\Delta\bi{\varepsilon}_{\rr{p}1},
  \dots,\Delta\bi{\varepsilon}_{\rr{p}m}) 
  &= \bi{0} , 
  \label{eq:fixed.1} \\
  \partial\delta_{Y_{l}}^{*}(\Delta\bi{\varepsilon}_{\rr{p}l}) 
  - \bi{\sigma}_{l}(\Delta\bi{\varepsilon}_{\rr{p}l},\Delta\bi{u}) 
  &\ni \bi{o} , 
  \quad l=1,\dots,m. 
  \label{eq:fixed.2}
\end{align}
We easily see that, 
for any $\alpha > 0$ and $\beta_{l} > 0$ $(l=1,\dots,m)$, 
\eqref{eq:fixed.1} and \eqref{eq:fixed.2} are satisfied if and only if 
\begin{align}
  \Delta\bi{u} 
  & = \Delta\bi{u} - \alpha 
  \bi{r}(\Delta\bi{u},\Delta\bi{\varepsilon}_{\rr{p}1},
  \dots,\Delta\bi{\varepsilon}_{\rr{p}m}) , 
  \label{eq:fixed.point.1} \\
  \Delta\bi{\varepsilon}_{\rr{p}l} 
  & \in \Delta\bi{\varepsilon}_{\rr{p}l} 
  - \beta_{l}
  (\partial\delta_{Y_{l}}^{*}(\Delta\bi{\varepsilon}_{\rr{p}l}) 
  - \bi{\sigma}_{l}(\Delta\bi{\varepsilon}_{\rr{p}l},\Delta\bi{u})) , \notag\\
  & \qquad\qquad\qquad\qquad 
  l=1,\dots,m 
  \label{eq:fixed.point.2}
\end{align}
hold. 
That is, the solution of the elastoplastic incremental problem is 
characterized as a fixed point of the mappings on the right sides of 
\eqref{eq:fixed.point.1} and \eqref{eq:fixed.point.2}. 

A natural fixed-point iteration applied to \eqref{eq:fixed.point.1} is 
\begin{align*}
  \Delta\bi{u}^{(k+1)} 
  := \Delta\bi{u}^{(k)} - \alpha 
  \bi{r}(\Delta\bi{u}^{(k)},\Delta\bi{\varepsilon}_{\rr{p}1}^{(k)},
  \dots,\Delta\bi{\varepsilon}_{\rr{p}m}^{(k)})  . 
\end{align*}
This is exactly same as the computations in 
lines~\ref{alg:proximal.plasticity.elastic}, 
\ref{alg:proximal.plasticity.stress}, 
\ref{alg:proximal.plasticity.residual}, and 
\ref{alg:proximal.plasticity.displacement}
of \refalg{alg:proximal.plasticity}. 

In contrast, we apply a slightly different scheme to \eqref{eq:fixed.point.2}. 
Namely, for each $l=1,\dots,m$ we rewrite \eqref{eq:fixed.point.2} 
equivalently as 
\begin{align}
  \Delta\bi{\varepsilon}_{\rr{p}l} 
  + \beta_{l} \partial\delta_{Y_{l}}^{*}(\Delta{\bi{\varepsilon}}_{\rr{p}l}) 
  \ni \Delta\bi{\varepsilon}_{\rr{p}l} 
  + \beta_{l} \bi{\sigma}_{l}(\Delta\bi{\varepsilon}_{\rr{p}l},\Delta\bi{u}) ,
  \label{eq:normality.proximal.4}
\end{align}
from which we obtain an iteration 
\begin{align}
  \Delta\bi{\varepsilon}_{\rr{p}l}^{(k+1)} 
  + \beta_{l} 
  \partial\delta_{Y_{l}}^{*}(\Delta{\bi{\varepsilon}}_{\rr{p}l}^{(k+1)}) 
  \ni \Delta\bi{\varepsilon}_{\rr{p}l}^{(k)} 
  + \beta_{l} \bi{\sigma}_{l}
  (\Delta\bi{\varepsilon}_{\rr{p}l}^{(k)}, \Delta\bi{u}^{(k)}) . 
  \label{eq:normality.proximal.0}
\end{align}
It follows from \reflmm{lmm:resolvent.interpretation} that 
\eqref{eq:normality.proximal.0} is equivalent to 
\begin{align}
  \Delta\bi{\varepsilon}_{\rr{p}l}^{(k+1)}
  = \prox_{\beta_{l} \delta_{Y_{l}}^{*}}
  (\Delta\bi{\varepsilon}_{\rr{p}l}^{(k)} + \beta_{l} \bi{\sigma}_{l}
  (\Delta\bi{\varepsilon}_{\rr{p}l}^{(k)}, \Delta\bi{u}^{(k)})  )  . 
  \label{eq:normality.proximal.2}
\end{align}
This corresponds to line~\ref{alg:proximal.plasticity.prox} of 
\refalg{alg:proximal.plasticity}. 
Since application of the proximal operator of a closed proper convex 
function to any point always results in a (nonempty) 
unique point \cite[section~1.1]{PB14}, 
the right side of \eqref{eq:normality.proximal.2} is 
guaranteed to be uniquely determined. 

\begin{figure}[tbp]
  \centering
  \includegraphics[scale=0.90]{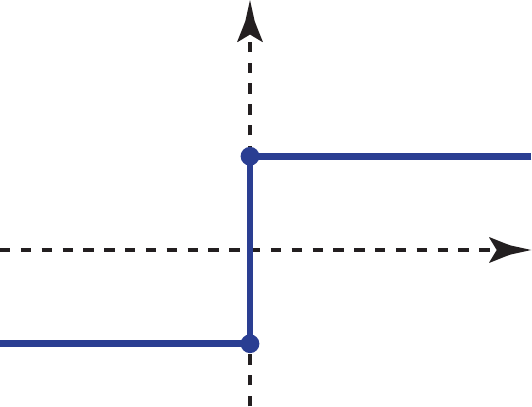}
  \begin{picture}(0,0)
    \put(-150,-50){
    \put(128,80){{\footnotesize $\Delta\varepsilon_{l}$ }}
    \put(63,146){{\footnotesize $z$ }}
    \put(104,120){{\footnotesize $z \in \partial\delta_{Y_{l}}^{*}(\Delta\varepsilon_{l})$ }}
    }
  \end{picture}
  \caption{The associated flow rule of \refex{ex:truss}. }
  \label{fig:stress_function}
\end{figure}

\begin{remark}
  We have seen that the update rule of the plastic strain, 
  \eqref{eq:normality.proximal.2}, of \refalg{alg:proximal.plasticity} 
  can be obtained not from \eqref{eq:fixed.point.2} 
  but from \eqref{eq:normality.proximal.4}.  
  If we adopt \eqref{eq:fixed.point.2}, application of a fixed-point 
  iteration yields 
  \begin{align}
    \Delta\bi{\varepsilon}_{\rr{p}l}^{(k+1)} 
    \in \Delta\bi{\varepsilon}_{\rr{p}l}^{(k)} 
    - \beta_{l} (\partial\delta_{Y_{l}}^{*}
    (\Delta\bi{\varepsilon}_{\rr{p}l}^{(k)}) - \bi{\sigma}_{l}
    (\Delta\bi{\varepsilon}_{\rr{p}l}^{(k)}, \Delta\bi{u}^{(k)}) ) . 
    \label{eq:steepest.plastic.1}
  \end{align}
  However, this is not adequate as an update rule, because 
  $\partial\delta_{Y_{l}}^{*}(\Delta\bi{\varepsilon}_{\rr{p}l}^{(k)})$ 
  on the right side of \eqref{eq:steepest.plastic.1} is, in general, not 
  determined uniquely. 
  For example, in the truss case considered in \refex{ex:truss}, 
  \reffig{fig:stress_function} shows 
  $\partial\delta_{Y_{l}}^{*}(\Delta\varepsilon_{\rr{p}l})$, where 
  $\delta_{Y_{l}}^{*}(\Delta\varepsilon_{\rr{p}l})$ is shown in  
  \reffig{fig:dissipation_function}. 
  It is observed in \reffig{fig:stress_function} that 
  $\partial\delta_{Y_{l}}^{*}(\Delta\varepsilon_{\rr{p}l})$ is not 
  unique at $\Delta\varepsilon_{\rr{p}l}=0$. 
  In contrast, as mentioned above, 
  $\Delta\bi{\varepsilon}_{\rr{p}l}^{(k+1)}$ satisfying 
  \eqref{eq:normality.proximal.0} exists uniquely, which makes 
  \refalg{alg:proximal.plasticity} well-defined. 
  \finbox
\end{remark}

\begin{remark}
  To provide another viewpoint, observe that 
  \eqref{eq:normality.proximal.0} is equivalently written as 
  \begin{align}
    \bi{\sigma}_{l}(\Delta\bi{\varepsilon}_{\rr{p}l}^{(k)}, \Delta\bi{u}^{(k)})
    - \frac{1}{\beta_{l}} 
    (\Delta\bi{\varepsilon}_{\rr{p}l}^{(k+1)} 
    - \Delta\bi{\varepsilon}_{\rr{p}l}^{(k)} )
    \in \partial\delta_{Y_{l}}^{*}(\Delta\bi{\varepsilon}_{\rr{p}l}^{(k+1)}) .
    \label{eq:normality.proximal.6}
  \end{align}
  This is analogous to the associated flow rule 
  in \eqref{eq:equilibrium.4}, 
  but the second term on the left side seems to be additional. 
  One may consider that a natural update rule based on 
  \eqref{eq:equilibrium.4} is 
  \begin{align}
    \bi{\sigma}_{l}(\Delta\bi{\varepsilon}_{\rr{p}l}^{(k)}, \Delta\bi{u}^{(k)}) 
    \in 
    \partial\delta_{Y_{l}}^{*}(\Delta\bi{\varepsilon}_{\rr{p}l}^{(k+1)}) . 
    \label{eq:non.existence.1}
  \end{align}
  However, \eqref{eq:non.existence.1} is not adequate, because for 
  $\bi{\sigma}_{l}(\Delta\bi{\varepsilon}_{\rr{p}l}^{(k)},\Delta\bi{u}^{(k)}) 
  \not\in Y_{l}$ there exists no 
  $\Delta\bi{\varepsilon}_{\rr{p}l}^{(k+1)}$ satisfying 
  \eqref{eq:non.existence.1}; see \reffig{fig:stress_function} for the 
  truss case. 
  In contrast, as mentioned above, for any 
  $\bi{\sigma}_{l}(\Delta\bi{\varepsilon}_{\rr{p}l}^{(k)},\Delta\bi{u}^{(k)})$ 
  and $\Delta\bi{\varepsilon}_{\rr{p}l}^{(k)}$, 
  $\Delta\bi{\varepsilon}_{\rr{p}l}^{(k+1)}$ satisfying 
  \eqref{eq:normality.proximal.6} exists uniquely. 
  \finbox
\end{remark}

\section{Conclusions}
\label{sec:conclusion}

This short paper has presented a unified form of proximal gradient method 
to solve quasi-static incremental problems in elastoplastic analysis of 
structures. 
An interpretation of the presented algorithm has also been provided from 
a viewpoint of mechanics. 
Although in this paper we have restricted ourselves to perfect 
plasticity, extension of the presented results to problems with strain 
hardening is possible; one can refer to \cite{Kan16,SK18,SK20} for 
cases with specific yield criteria. 

The presented general form of the algorithm, as well as interpretation, 
sheds new light on numerical methods in computational plasticity. 
For example, although this paper has been restricted to quasi-static 
incremental problems, it can possibly provide us with a guide for 
development of similar algorithms solving other problems in plasticity, 
including, e.g., limit analysis and shakedown analysis. 
Such algorithms combined with an acceleration scheme may possibly be 
efficient compared with conventional methods in plasticity, because it 
has been reported for quasi-static incremental problems that accelerated 
proximal gradient methods outperform conventional optimization-based 
approaches, especially for large-scale problems \cite{Kan16,SK18,SK20}.

\paragraph{Acknowledgments}

The work described in this paper is partially supported by 
JSPS KAKENHI 17K06633 
and 
JST CREST Grant Number JPMJCR1911, Japan.

\end{document}